\documentclass[11pt,letterpaper]{article}
\usepackage{amsfonts, amsmath, amssymb, amscd, amsthm, color, mathrsfs, wasysym, setspace, mdwlist, calc, float, mathtools, tikz-cd}

\usepackage[utf8]{inputenc}

\usepackage[english]{babel}

\usepackage[top=3cm,bottom=3cm,left=4cm,right=2cm,marginparwidth=1.75cm]{geometry}

\usepackage[all,arc]{xy}
\usepackage{enumerate}
\usepackage{mathrsfs}
\usepackage{enumitem}

\newtheorem{thm}{Theorem}[section]
\newtheorem{cor}[thm]{Corollary}
\newtheorem{proposition}[thm]{Proposition}
\newtheorem{lem}[thm]{Lemma}

\theoremstyle{definition}
\newtheorem{defn}[thm]{Definition}

\theoremstyle{remark}

\newtheorem{example}[thm]{Example}

\makeatletter
\let\c@equation\c@thm
\numberwithin{equation}{section}
\makeatother

\usepackage{amsmath}
\usepackage{amssymb}
\usepackage{mathtools}
\usepackage[colorlinks=true, allcolors=blue]{hyperref}
\usepackage{amsfonts}
\usepackage[mathscr]{euscript}
\usepackage{graphicx}
\begin{document}
\title{The Geometry of Subgroup Embeddings and Asymptotic Cones}
\author{Andy Jarnevic}
\maketitle
\begin{abstract}
Given a finitely generated subgroup $H$ of a finitely generated group $G$ and a non-principal ultrafilter $\omega$, we consider a natural subspace, $Cone^{\omega}_{G}(H)$, of the asymptotic cone of $G$ corresponding to $H$. Informally, this subspace consists of the points of the asymptotic cone of $G$ represented by elements of the ultrapower $H^{\omega}$. We show that the connectedness and convexity of $Cone^{\omega}_{G}(H)$ detect natural properties of the embedding of $H$ in $G$. We begin by defining a generalization of the distortion function and show that this function determines whether $Cone^{\omega}_{G}(H)$ is connected. We then show that whether $H$ is strongly quasi-convex in $G$ is detected by a natural convexity property of $Cone^{\omega}_{G}(H)$ in the asymptotic cone of $G$.
\end{abstract}
 
\tableofcontents 

\section{Introduction}

The asymptotic cone of a group $G$ is a metric space which captures certain aspects of the coarse geometry of $G$. Roughly speaking, the asymptotic cone is how the group looks from infinitely far away, and is constructed by taking a certain limit of scaled down copies of the group viewed as a metric space. The roots of asymptotic cones come from a paper of Gromov proving that finitely generated groups of polynomial growth are nilpotent \cite{Gromov1}. Van den Dries and Wilkie added non-standard analysis to the construction in this paper, formally introducing asymptotic cones \cite{VDW}. Since then, several other standard algebraic and geometric properties of groups have been shown to have natural parallels in their asymptotic cones. For instance, a finitely generated group is  virtually abelian if and only if all of its asymptotic cones are quasi-isometric to $\mathbb{R}^{n}$ for some $n \in \mathbb{N}$ \cite{Gromov2}, and a finitely-generated group is hyperbolic if and only if all of its asymptotic cones are $\mathbb{R}$-trees \cite{Gromov2}. 

Given a group $G$ and an ultrafilter $\omega$, we will denote the asymptotic cone of $G$ with respect to $\omega$ by $Cone^{\omega}(G)$. The goal of this paper is to study the way that geometric properties of embeddings of subgroups in groups can be detected using asymptotic cones. In order to accomplish this, we define a natural subspace of $Cone^{\omega}(G)$ corresponding to a subgroup $H$. Essentially, points in the asymptotic cone of a group $G$ can be represented by certain elements of the ultrapower $G^{\omega}$. We denote by $Cone^{\omega}_{G}(H)$ the subspace of $Cone^{\omega}(G)$ consisting of points with a representative from $H^{\omega}$. For the formal definition of this subspace, see Definiton 4.10.

\par
The first property of $Cone^{\omega}_{G}(H)$ we study is connectedness. We show that whether $Cone^{\omega}_{G}(H)$ is connected is closely related to a generalization of the distortion function of $H$ in $G$.

\begin{defn}
Let $H$ be a subgroup of a group $G$, with $G = \langle X \rangle$  and $H = \langle Y \rangle$ where $X$ and $Y$ are finite sets. The $\textit{distortion}\textrm{ } \textit{function}$ of $H$ in $G$ with respect to $X$ and $Y$ is defined by the formula $$\Delta^{G,X}_{H,Y}(n) = \textrm{max}\{|h|_{Y} \mid h \in H,  \, |h|_{X} \leq n\},$$ where $|h|_{Y}$ denotes the word length of $h$ with respect to the generating set $Y$. A subgroup $H$ of a group $G$ is called $\textit{undistorted}$ if $\Delta^{G,X}_{H,Y}$ is bounded from above by a linear function.
\end{defn}

We consider distortion up to the following equivalence relation.

\begin{defn} For non-decreasing functions $f,g\colon \mathbb{N} \rightarrow \mathbb{N}$, we write that $f \preceq g$ if there exists a constant $C$ such that $f(n) \leq Cg(Cn)$ for all $n \in \mathbb{N}$. We write $f \sim g$ if $f \preceq g$ and $g \preceq f$. \end{defn}

Under this equivalence, distortion is independent of the choice of the finite generating set.

\begin{defn}
Assume that $X$ is a finite generating set for a group $G$, and $H$ is a subgroup of $G$ such that $X$ contains a generating set for $H$. We define the $\textit{generalized distortion function}$, $\mu_{H}^{G,X}(m,n)\colon \mathbb{N} \times \mathbb{N} \rightarrow \mathbb{R}$ by the formula $$\mu_{H}^{G,X}(m,n) = \textrm{max}\{|h|_{Y_{m}}\mid h \in H, |h|_{X} \leq n\} =  \Delta^{G,X}_{H,Y_{m}}(n)$$ where $Y_{m} = \{h \in H \mid |h|_{X} \leq m\}$. 
\end{defn}

We consider generalized distortion functions up to the following equivalence.
\begin{defn}
Given two functions $f,g\colon \mathbb{N} \times \mathbb{N}\rightarrow \mathbb{R}$ which are non-increasing in the first variable, and non-decreasing in the second variable, we write $f \preceq g$ if there exists a constant $C \in \mathbb{N}$ such that $$f(Cm,n) \leq Cg(m,Cn)+C$$ for all $m,n \in \mathbb{N},$ and we say that $f \cong g$ if $f \preceq g$ and $g \preceq f$.
\end{defn}
Under this equivalence, $\mu_{H}^{G,X}(n)$ is independent of the choice of  the finite generating set $X$ of $G$, so we use $\mu^{G}_{H}$ to mean $\mu^{G,X}_{H}$ where $X$ is some finite generating set of $G$. For example, if $H$ is undistorted in $G$, then $$\mu^{G}_{H}(m,n) \cong  \frac{n}{m}.$$ 

We show that the generalized distortion function determines whether $Cone^{\omega}_{G}(H)$ is connected. Specifically, we prove the following result, which also shows that for such a subspace connectedness is equivalent to path-connectedness.

\begin{defn}
We say that a function $f\colon\mathbb{R}^{\geq 1} \times \mathbb{R}^{\geq 0} \rightarrow \mathbb{R}$ is \textit{homogeneous} if $f(r,s) = g(\frac{s}{r})$ for some function $g\colon \mathbb{R}^{\geq 0} \rightarrow \mathbb{N}$.
\end{defn}
\begin{thm} (Theorem 4.13)
    For any finitely generated group $G$ and any subgroup $H$, the following conditions are equivalent.
\begin{enumerate}

\item $H$ is finitely generated and $\mu^{G}_{H}(m,n)$ is bounded from above by a homogeneous function.

\item $Cone_{G}^{\omega}(H)$ is path connected for all non-principal ultrafilters $\omega$.

\item $Cone_{G}^{\omega}(H)$ is connected for all non-principal-ultrafilters $\omega$.
\end{enumerate}
\end{thm}

This theorem enables us to relate the ordinary distortion function to the connectedness of $Cone^{\omega}_{G}(H),$ and to construct pairs $H \leq G$  such that $Cone^{\omega}_{G}(H)$ is disconnected, but the distortion of $H$ in $G$ is small. Consider the following properties of a finitely generated subgroup $H$ of a finitely generated group $G$:
\begin{enumerate}[label=(\alph*)] \item $H$ is undistorted in $G$, \item $Cone^{\omega}_{G}(H)$ is connected for all non-principal ultrafilters $\omega$, \item $\Delta^{G}_{H}$ is bounded by a polynomial function. 
\end{enumerate}
The following theorem collects the relationship between these three properties.
\begin{thm} (Theorem 4.19)
For any finitely generated subgroup $H$ of a finitely generated group $G$, the following implications hold:
$$(a) \Rightarrow (b) \Rightarrow (c)$$ 
Further, the missing implications do not hold. Specifically, we have the following. \begin{enumerate} \item  For any $k \in \mathbb{N}$, there exists a finitely generated group $G$ and a finitely generated subgroup $H$ of $G$ such that $\Delta^{G}_{H}(n) \sim n^{k}$ and $Cone^{\omega}_{G}(H)$ is connected for any non-principal ultrafilter $\omega$.  \item For any real number $\epsilon > 0$, there exists a finitely generated group $G$ with a finitely generated subgroup $H$ such that $\Delta^{G}_{H}(n) \preceq n^{1+\epsilon}$ but $Cone^{\omega}_{G}(H)$ is disconnected for some non-principal ultrafilter $\omega$. 
\end{enumerate}
\end{thm}

\par
Next, we show that the property of a subgroup being strongly quasi-convex, introduced independently by Tran and Genevois \cite{Morse, quasiconvex}, can be detected by a natural property of the embedding of $Cone^{\omega}_{G}(H)$ in $Cone^{\omega}(G)$.  

\begin{defn} A subgroup $H$ of a group $G$ with finite generating set $X$ is said to be \textit{quasi-convex} if there exists a number $M$ such that any geodesic in the Cayley graph $\Gamma(G,X)$ connecting two points in $H$ is contained in the $M$ neighborhood of $H$. $H$ is said to be $\textit{strongly quasi-convex}$ if for all real numbers $\lambda \geq 1,C \geq 0$ there exists a constant $N(\lambda,C)$ such that  any $(\lambda,C)$-quasi-geodesic in $\Gamma(G,X)$ connecting two points in $H$ is entirely contained in the $N$ neighborhood of $H$.
\end{defn}

In general, quasi-convexity is not independent of the choice of the finite generating set of $G$. For instance, in the group $\mathbb{Z} \times \mathbb{Z} = \langle a \rangle \times \langle b \rangle$, the subgroup $\langle ab \rangle$ is not quasi-convex with respect to the generating set $\langle a, b \rangle$, but is quasi-convex with respect to the generating set $\langle ab, a \rangle$. In the case where $G$ is hyperbolic, quasi-convexity is independent of the choice of the finite generating set. 

We have the following relationship between these properties of a subgroup $H$ of a finitely generated group $G$: $$ \textrm{strongly quasi-convex} \Rightarrow \textrm{quasi-convex} \Rightarrow \textrm{finitely generated and undistorted}.$$
None of the reverse implications hold. To see this again consider $G = \mathbb{Z} \times \mathbb{Z} = \langle a \rangle \times \langle b \rangle$. The subgroup $\langle ab \rangle$ is undistorted but not quasi-convex, and the subgroup $\langle a \rangle$ is quasi-convex but not strongly quasi-convex. However, in the case when $G$ is hyperbolic, all of these properties are in fact equivalent.

Strong quasi-convexity is a generalization of quasi-convexity that is preserved under quasi-isometry in general. Tran \cite{quasiconvex} characterized strongly quasi-convex subgroups based on a certain divergence function, and showed that they satisfy many properties of quasi-convex sugroups of hyperbolic groups. Specifically, any strongly quasi-convex subgroup is undistorted, has finite index in its commensurator, and the intersection of any two strongly quasi-convex subgroups is strongly quasi-convex. Examples of strongly quasi-convex subgroups include peripheral subgroups of relatively hyperbolic groups and hyperbolically embedded subgroups of finitely generated groups.

We show that the property of being strongly quasi-convex is equivalent to a natural property of the embedding of $Cone^{\omega}_{G}(H)$ in $Cone^{\omega}(G)$.

\begin{defn}
We say that a subspace $T$ of a metric space $S$ is \textit{strongly convex} if any simple path in $S$ starting and ending in $T$ is entirely contained in $T$.
\end{defn}

\begin{thm} (Theorem 5.12)
Let $H$ be a finitely generated subgroup of a finitely generated group $G$. $H$ is strongly quasi-convex in $G$ if and only if $Cone^{\omega}_{G}(H)$ is strongly convex in $Cone^{\omega}(G)$ for all non-principal ultrafilters $\omega$.
\end{thm}

This characterization gives useful information about the structure of the asymptotic cones of groups with strongly quasi-convex subgroups. For instance, we obtain the following result.

\begin{thm} (Theorem 5.13) If $G$ is a finitely generated group containing an infinite, infinite index strongly quasi-convex subgroup $H$, then all asymptotic cones of $G$ contain a cut point.
\end{thm}

A precursor to Theorems 1.10 and 1.11 can be found in \cite{Behrstock}, where Behrstock showed that any asymptotic cone of a mapping class group contains an isometrically embedded copy of an $\mathbb{R}$ tree, and that this $\mathbb{R}$ tree is strongly convex in the asymptotic cone. This is then used to deduce that any asymptotic cone of a mapping class group contains a cut point. I would like to thank Jason Behrstock for pointing out this connection.

Combining Theorem 1.11 with a result of Drutu and Sapir \cite{trees} gives the following result.

\begin{cor} (Corollary 5.15)
If $G$ is a finitely-generated group containing an infinite, infinite index strongly quasi-convex subgroup, then $G$ does not satisfy a law.
\end{cor}

This result can be applied to show for instance that solvable groups and groups satisfying the law $x^{n} = 1$ for some $n \in \mathbb{N}$ cannot have infinite, infinite index strongly quasi-convex subgroups.

The paper is organized as follows. Section 2 covers some necessary background on asymptotic cones and establishes our notation. Section 3 establishes some basic properties of the generalized distortion function and formulates a relationship between the generalized distortion function and the distortion function. Section 4 contains the proof of Theorems 1.6 and 1.7. Finally, section 5 contains the proof of Theorems 1.10 and 1.11.

\section{Background}
In this section, we provide some background and fix our notation for asymptotic cones.

Recall that given an ultrafilter $\omega$ and any bounded sequence of real numbers, $(r_{i})$, $\lim^{\omega} (r_{i})$ exists and is unique.

Now let $(S,d)$ be a metric space, and let $c_{i}$ be an unbounded, strictly increasing sequence of positive real numbers. Denote by $d_{i}$ the metric on $S$ defined by $d_{i}(x,y) = d(x,y)/c_{i}.$ We call the sequence $(c_{i})$ the $\textit{scaling sequence}$.

\begin{defn}
    Given a metric space $(S,d)$, a scaling sequence $(c_{i})$, and an infinite sequence of points $z = (s_{i})$ in $S$, denote by $S^{\mathbb{N}}_{z}$ the set of infinite sequences $(t_{i})$ in $S$ such that $d_{i}(s_{i},t_{i})$ is bounded. The sequence $(s_{i})$ is called the $\textit{observation point}$. 
\end{defn}

\begin{defn} 
    Given $(x_{i}),(y_{i}) \in S^{\mathbb{N}}_{z}$, let $d^{*}((x_{i}),(y_{i})) = \lim^{\omega}d_{i}(x_{i},y_{i})$.
\end{defn}

Note that this is a bounded sequence so the limit exists. However, in general $d^{*}$ will not be a metric, as there can be different sequences $(x_{i}),(y_{i})$ such that $d^{*}((x_{i}),(y_{i})) = 0$.

\begin{defn}
    We will denote by $Cone^{\omega}_{z}((d_{i}),S)$ the metric space that results from quotienting the pseudo-metric $d^{*}$ by the equivalence relation $(x_{i}) \sim (y_{i})$ if $d^{*}((x_{i}),(y_{i})) = 0$. We will denote the resultant metric by $d^{\omega}_{S}$. When the choice of the base point or the scaling sequence is clear, we will simply write $Cone^{\omega}(S)$. We will denote the equivalence class of $(x_{i})$ by $(x_{i})^{\omega}$, so $d^{\omega}_{S}((x_{i})^{\omega},(y_{i})^{\omega}) = d^{*}((x_{i}),(y_{i})).$
\end{defn}

\begin{defn}
A map $f$ between two metric spaces $(S,d_{S})$ and $(T,d_{T})$ is called a                $(\lambda,C)\textit{-quasi-isometric embedding}$ if for all $s,t \in S$ 
$$\frac{d_{S}(s,t)}{\lambda}-C \leq d_{T}(f(s),f(t)) \leq \lambda d_{S}(s,t)+C.$$ $f$ is called $\epsilon\textit{-quasi-surjective}$ if for all $t \in T$, there exists an $s \in S$ such that $d_{T}(f(s),t) \leq \epsilon$. A map $f$ is called a $(\lambda,C,\epsilon)\textit{-quasi-isometry}$ if $f$ is a $(\lambda,C)$-quasi-isometric embedding, and is $\epsilon$-quasi-surjective. When we don't care about the quasi-isometry constants, we will simply call $f$ a quasi-isometry and say that $S$ and $T$ are quasi-isometric.
\end{defn}

\begin{defn}
Let $S$ be a metric space. A path $p\colon [0,\ell] \rightarrow S$ is called a $(\lambda,C)\textit{-quasi-geodesic}$ if $p$ is a $(\lambda,C)$-quasi-isometric embedding.
\end{defn}

\begin{defn} 
Given a pointed metric space $(S,x)$ and $(\lambda,C)$-quasi-geodesic paths \newline $p_{i} \colon [0,\ell_{i}] \rightarrow S$ such that the sequence $\ell_{i}/c_{i}$ is bounded and $(p_{i}(0)) \in S_{z}^{\mathbb{N}}$, let $L = \textrm{lim}^{\omega} \ell_{i}/c_{i}$. If $L \neq 0$, define the $\omega\textit{-limit}$ of the paths $p_{i}$, denoted $$p = \textrm{lim}^{\omega}(p_{i}) \colon [0,L] \rightarrow Cone^{\omega}(S),$$ by the following formula: $p(x) = \left(p_{i}\left(x\frac{\ell_{i}}{L}\right)\right)^{\omega}$. If $L = 0$, define $p = \textrm{lim}^{\omega}(p_{i})\colon \{0\} \rightarrow Cone^{\omega}(S)$ by the formula $p(0) = (p_{i}(0))^{\omega}$.
\end{defn}

\begin{defn} A geodesic in $Cone^{\omega}(S)$ is called a $\textit{limit geodesic}$ if it is an $\omega$-limit of geodesic paths.
\end{defn}

Note that the limit of geodesics is a geodesic in the asymptotic cone. Thus, if $S$ is a geodesic metric space, then so is $Cone^{\omega}(S)$.

A finitely generated group $G$ can be considered as a metric space using the word metric arising from any finite generating set $X$. Given an ultrafilter $\omega$, we will denote the asymptotic cone of $G$ with respect to $\omega$ by $Cone^{\omega}(G)$ where we assume all scaling sequences are $c_{i}=i$ unless otherwise specified, and the observation point will always be $(e)^{\omega}$. Note that $G$ is $(0,0,\frac{1}{2})$ quasi-isometric to its Cayley graph $\Gamma(G,X)$, and so its asymptotic cone is isometric to the asymptotic cone of $\Gamma(G,X)$. This is a geodesic space, and so we have that $Cone^{\omega}(G)$ is a geodesic space. 

The asymptotic cone of $G$ depends on the choice of a finite generating set $X$, an ultrafilter $\omega$,  and the choice of a scaling sequence $(d_{i})$. Note that changing the generating set of a group gives a quasi-isometric Cayley graph, and so will give a quasi-isometric asymptotic cone. In general, however, the other choices can matter, and a group can have many different asymptotic cones. For instance, Thomas and Velickovic exhibited a group such that one of its asymptotic cones is an $\mathbb{R}$-tree, and another is not simply connected \cite{ultrafilter}. These two choices turn out to be closely related. Specifically, given any scaling sequence $(c_{i})$ such that the sizes of the sets $S_{r} = \{i | c_{i} \in [r,r+1)\}$ are bounded, and any ultrafilter $\omega$, there exists an ultrafilter $\omega'$ such that $Cone^{\omega}((c_{i}),G) = Cone^{\omega'}((i),G)$ \cite{scaling}. This justifies our choice to take all scaling sequences as $c_{i}=i$ unless otherwise specified.

\begin{defn}
We say that a metric space $S$ is transitive if for any two points $s,t \in S$ there exists an isometry $\phi \colon S \rightarrow S$ such that $\phi(s) = t$.
\end{defn}

Recall that for any group $G$, $Cone^{\omega}(G)$ is a transitive space, and that any asymptotic cone is complete.

\section{The generalized distortion function}

We begin by defining a variant of distortion that will help us calculate generalized distortion in a variety of groups.

\begin{defn} 
Let $H$ be a subgroup of a group $G$ and let $Y,X$ be finite generating sets of $H$ and $G$ respectively. Define the $\textit{lower distortion function}$ of $H$ in $G$, denoted $\nabla^{G,X}_{H,Y}(n)$, by the formula $$\nabla^{G,X}_{H,Y}(n) = \min\{|h|_{Y} \mid |h|_{X} > n, h \in H\}.$$
\end{defn}

We consider lower distortion up to the same equivalence as distortion, and denote by $\nabla^{G}_{H}$ the function $\nabla^{G,X}_{H,Y}$ for some choices of the finite generating sets $X,Y$.
\begin{example}
For $p \in \mathbb{N}, p \geq 2$, let $G = BS(1,p) = \langle a,b|b^{-1}ab=a^{p}\rangle$, and let $H = \langle a\rangle$. Note that $a^{p^{n}} = b^{-n}ab^{n}$, and so $\Delta^{G}_{H}(n) \succeq p^{n}$. In fact, $\Delta^{G}_{H} \sim p^{n}$\cite{Gromov2}. Next, note that if $k < p^{n}$, then we can write $k = \sum_{i=0}^{n-1}c_{i}p^{i},$ with $0 \leq c_{i} < p.$ This in turn means that we can write $a^{k}=\prod_{i=0}^{n-1}b^{-i}a^{c_{i}}b^{i} = b^{-1}(\prod_{i=0}^{n-1}a^{c_{i}}b^{-1})b^{n-1}.$ This implies that $|a^{k}|_{X} \leq n + n(p) = n(p+1).$ Thus, $\nabla^{G}_{H}(n) \succeq p^{n}$.
\end{example}

\begin{example}
Let $G$ be the discrete Heisenberg group, i.e. the group of all upper triangular integer matrices with ones along the diagonal, and let $H$ be the center of this group, i.e. the subgroup of all matrices of the form $\begin{pmatrix}
1 & 0 & c \\
0 & 1 & 0 \\
0 & 0 & 1
\end{pmatrix}$ with $c \in \mathbb{Z}.$
 Let $X$ be the generating set for the group $G$ given by $G = \langle x,y,z\rangle $ where $x = \begin{pmatrix}
1 & 1 & 0 \\
0 & 1 & 0 \\
0 & 0 & 1
\end{pmatrix}$, $y = \begin{pmatrix}1 & 0 & 0 \\
0 & 1 & 1 \\
0 & 0 & 1\end{pmatrix}$, and $z = \begin{pmatrix}1 & 0 & 1 \\
0 & 1 & 0 \\
0 & 0 & 1\end{pmatrix}$, and let $Y = \{ z\}$, a generating set for $H$.
Note that $x^{n}y^{n}x^{-n}y^{-n} = z^{n^{2}}.$ Now let $m$ be a natural number such that $(n-1)^{2} < m < n^{2}.$ We know that $|z^{n^{2}}|_{X} \leq 4n$. Thus, $$|z^{m}|_{X} \leq 4n + (n^{2}-(n-1)^{2}) = 4n + 2n -1 \leq 6n.$$ Thus, if $m \leq n^{2}$, then $|z^{m}|_{X} \leq 6n$, and so $\nabla^{G}_{H}(n) \succeq n^{2}$.

Now we will show that if $|h|_{X} \leq n$, then $|h|_{Y} \leq n^{2}$. Let $f \colon G \rightarrow \mathbb{N}$ be the function given by $f\begin{pmatrix}1 & a & b \\
0 & 1 & c \\
0 & 0 & 1\end{pmatrix} = |a|$, and let $k \colon G \rightarrow \mathbb{N}$ be the function given by $k\begin{pmatrix}1 & a & b \\
0 & 1 & c \\
0 & 0 & 1\end{pmatrix}= |b|$. We have that $$f(gx) \leq f(g) + 1, \, f(gy) = f(g),\,  f(gz)=f(g),$$ and thus if $|g|_{X} \leq n$, then $f(g) \leq n$. Similarly, $$k(gx) = k(g), \, k(gy) \leq f(g)+k(g), \, k(gz) \leq k(g) + 1.$$ Thus if $|g|_{X} \leq n$, then $k(g) \leq  n^{2}.$ If $h \in H$, then $|h|_{Y} = k(h)$, and so if $|h|_{X} \leq n$, then $|h|_{Y} \leq n^{2}$. Thus, $\Delta^{G}_{H}(n) \preceq n^{2}$.
\end{example}

\begin{example}
Let $G = \langle a,b,c | [a,b] =1, [a,c] = 1, c^{-1}bc = b^{2}\rangle \cong \mathbb{Z} \times BS(1,2)$, and let $H = \langle a,b \rangle \cong \mathbb{Z} \times \mathbb{Z}$. Let $X = \{a,b,c\}$. Note that $|b^{2^{n}}|_{X} \leq 2n+1$, so $\Delta^{G}_{H}(n) \succeq 2^{n}$, but $|a^{n}|_{X} =n$, and so $\nabla^{G}_{H}(n) \preceq n$. Thus, we have that $\Delta^{G}_{H} \not  \sim \nabla^{G}_{H}$.
\end{example}
Note that if $f_{1}, \, f_{2}, \, g_{1}$ and $g_{2}$ are strictly increasing functions such that $f_{1}(n) \sim f_{2}(n)$ and $g_{1}(n) \sim g_{2}(n)$ then $f_{1}(n)/g_{1}(m) \cong f_{2}(n)/g_{2}(m)$. Thus, we can state the following proposition. 

\begin{proposition}
For a finitely generated subgroup $H$ of a finitely generated group group $G$, the following inequalities hold 

\begin{equation} 
\tag{1} \label{eqn:b} \frac{\Delta^{G}_{H}(n)}{\Delta_{H}^{G}(m)} \preceq \mu^{G}_{H}(m,n) \preceq \frac{\Delta^{G}_{H}(n)}{\nabla^{G}_{H}(m)}.
\end{equation} 

\end{proposition}

\begin{proof}
First, choose a finite generating set $X$ for $G$ containing a generating set $Y$ for $H$. Fix $n \in \mathbb{N}$ and let $h$ be an element of H such that $|h|_{X} \leq n$, and $|h|_{Y} = \Delta^{G,X}_{H,Y}(n)$. By definition, if $k \in Y_{m}$ then $|k|_{X} \leq m$, and so $|k|_{Y} \leq \Delta_{H,Y}^{G,X}(m)$. Thus, $|h|_{Y_{m}} \geq \left \lceil \Delta^{G,X}_{H,Y}(n)/ \Delta^{G,X}_{H,Y}(m)\right \rceil,$ and we obtain the first inequality in $\eqref{eqn:b}.$ For the next inequality, note that if $|h|_{X} \leq n$, then $|h|_{Y} \leq \Delta^{G,X}_{H,Y}(n)$. Thus, we can write $h$ as a product of at most $\left \lceil \Delta^{G,X}_{H,Y}(n)/ (\nabla^{G,X}_{H,Y}(m)-1)\right \rceil$ elements of length less than or equal to $\nabla^{G,X}_{H,Y}(m)-1$ with respect to $Y$. Note that if $h$ is an element of $H$ such that $|h|_{Y} < \nabla^{G,X}_{H,Y}(m)$, then by the definition of $\nabla^{G,X}_{H,Y}$, $|h|_{X} \leq m$, and $h \in Y_{m}$.  This gives the second inequality in $\eqref{eqn:b}$.
\end{proof}

\begin{defn}
We call a subgroup $H$ of a group $G$ $\textit{uniformly distorted}$ if $\Delta^{G}_{H} \sim \nabla^{G}_{H}$.
\end{defn}

Combining the previous observations gives the following corollary.
\begin{cor}
If $H$ is a uniformly distorted finite subgroup of a finite group $G$, then $\mu^{G}_{H}(m,n) \cong \frac{\Delta^{G}_{H}(n)}{\Delta^{G}_{H}(m)}\cong\frac{\Delta^{G}_{H}(n)}{\nabla^{G}_{H}(m)}.$
\end{cor}

\begin{example}
Example 3.2 showed that if $G = BS(1,p) = \langle a,b \mid b^{-1}ab = a^{p}\rangle$ and $H = \langle a \rangle$, then $H$ is uniformly distorted in $G$, so we can apply Corollary 3.7 to get that $\mu^{G}_{H}(m,n) \cong p^{n-m}.$
\end{example}
\begin{example}
Example 3.3 showed that if $G$ is the discrete Heisenberg group, and $H$ is the center of $G$ then $H$ is uniformly distorted in $G$ and we have from Corollary 3.7 that $\mu^{G}_{H}(m,n) \cong \left(n/m\right)^{2}.$
\end{example}

We conclude with an example demonstrating that for a group $G$ with finitely generating set $X$ containning a generating set for a subgroup $H$, $\mu^{G,X}_{H}(n-1,n)$ can be very large.

\begin{example}
Let $H$ be a finitely generated subgroup of a finitely generated group $G$ such that the membership problem is undecidable, and let $X$ be a finite generating set for $G$ containing a generating set of $H$. The existence of such subgroups was demonstrated independently by Mihailova and Rips \cite{cancellation} \cite{russian}. Gromov \cite{Gromov2} showed that the distortion function of $H$ in $G$ is bounded by a computable function if and only if the membership problem is solvable. Note that $\Delta_{H,Y}^{G,X}(n) = \mu_{H}^{G,X}(1,n) \leq \mu_{H}^{G,X}(1,2) \mu_{H}^{G,X}(2,3) \, \dots \, \mu_{H}^{G,X}(n-1,n)$. Thus, if $\mu_{H}^{G,X}(n-1,n)$ is bounded by a computable function, then so is $\Delta^{G,X}_{H,Y}(n)$, a contradiction. Thus, $\mu_{H}^{G,X}(n-1,n)$ is not bounded by any computable function.
\end{example}

\section{Connectedness in asymptotic cones}

We begin by defining an analog of the generalized distortion function for the case of a metric space $S$. 

\begin{defn}
Given a metric space $S$, a real number $r > 0$, and two points $s,t \in S$, an $\textit{r-path}$ connecting $s$ and $t$ is a sequence of points $s=s_{0}, \, s_{1}, \, \dots \, ,s_{k} = t$ with $d_{S}(s_{i},s_{i+1}) \leq r$ for all $ 0 \leq i < k$. We call $k$ the $\textit{length}$ of the $r$-path. We say a metric space $S$ is $\textit{r-connected}$ if for any two points $s,t \in S$ there exists an $r$-path connecting $s$ and $t$. If $(S,s)$ is a pointed $r$-connected metric space, and $t$ is in $S$, let $|t|_{r}$ be the length of the shortest $r$-path connecting $s$ and $t$.
\end{defn}

\begin{defn} Let $(S,s)$ be a proper r-connected pointed metric space. Define $\nu_{S}(m,n) \colon \mathbb{R}^{\geq r} \times \mathbb{R}^{\geq 0} \rightarrow \mathbb{N}$ to be $\textrm{max}\{|t|_{m}\mid d_{S}(s,t) \leq n\}.$ \end{defn}

\begin{lem}
$\nu_{S}$ is well-defined, i.e. for all real numbers $m \geq r ,n$ there exists a constant $K \in \mathbb{R}$ such that for any point $t \in S$ with $d(s,t) \leq n$, $|t|_{m} \leq K$.
\end{lem}

\begin{proof}
Fix $n \in \mathbb{R}^{\geq 0}$, and let $B$ be the closed ball centered at $s$ of radius $n$. As $B$ is compact, it can be covered by some finite number $p$ of open balls of radius $m$. Let $s_{1},\dots s_{p}$ be the centers of these balls. As $S$ is $r$-connected for each $s_{i}$ there exists a sequence of points $$s = s_{0,i}, \, s_{1,i}, \dots,\,  s_{K_{i},i} = s_{i}$$ with $d_{S}(s_{j,i},s_{j+1,i}) \leq m$ for all $0 \leq i < K_{i}.$ Let $K = \textrm{max} \{K_{i}\mid 1 \leq i \leq p\}$. Any point in $B$ is within $m$ of some $s_{i}$, and so $\nu_{S}(m,n) \leq K + 1$.
\end{proof}
If $H$ is a finitely-generated subgroup of a finitely generated group $G$, and $X$ is a finite generating set fo $G$ containing a generating set for $H$, then $H$ is 1-connected and proper with respect to the word metric induced by $X$. It is clear in this case that $\mu^{G}_{H}$ is the restriction of $\nu_{H}$ to $\mathbb{N} \times \mathbb{N}$, where we consider $H$ with the word metric induced from $G$. 

\begin{defn}
Given two functions $f,g \colon \mathbb{R}^{\geq r} \times \mathbb{R}^{ \geq 0}\rightarrow \mathbb{R}$ which are non-increasing in the first variable, and non-decreasing in the second variable, we write $f \preceq g$ if there exists a constant $C \in \mathbb{R}$ such that $f(Cm,n) \leq Cg(m,Cn)$ for all $m,n \in \mathbb{R}^{\geq 0}, m \geq r$ and we say that $f \cong g$ if $f \preceq g$ and $g \preceq f$.
\end{defn}

Essentially, $\nu$ measures how far away $S$ is from being a geodesic metric space. For instance, if $S$ is geodesic, then $\nu_{S}(m,n) = \left \lceil n/m \right \rceil$.

\begin{figure}[h]
\centering
\includegraphics[width=0.7\textwidth]{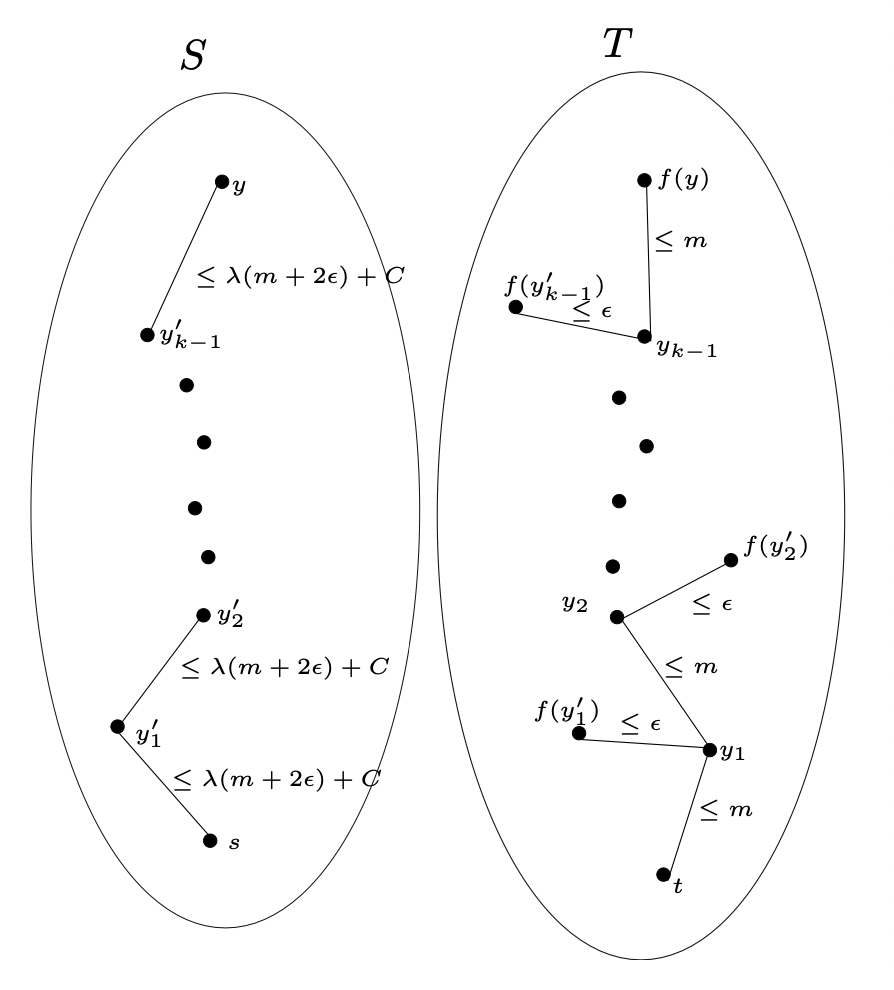}
\caption{Lemma 4.5}
\end{figure}

\begin{lem} If $(S,s),(T,t)$ are proper, r-connected pointed metric spaces, and $f$ is a $(\lambda,C,\epsilon)$-quasi-isometry between $S$ and $T$ such that $f(s) = t,$ then, $\nu_{S} \cong \nu_{T}$.
\end{lem}

\begin{proof}
First, fix $n \in \mathbb{R}^{\geq 0}, m \in \mathbb{R}^{\geq r}$, and let $y \in S$ with $d_{S}(s,y) \leq n$. This implies that $d_{T}(t,f(y)) \leq \lambda n + C$. Let $K = \nu_{T}(m,\lambda n+C)$. There exist $K + 1$ points $y_{0},\, y_{1} \, \dots y_{K}$ such that $t=y_{0}, \, y_{1}, \dots, \, y_{K}=f(y)$ with $d_{T}(y_{i},y_{i+1}) \leq m$. By quasi-surjectivity, for each $i$ there exists an $y_{i}' \in S$ such that $d_{T}(f(y_{i}'),y_{i}) \leq \epsilon$. Thus, $d_{T}(f(y_{i}'),f(y_{i+1}')) \leq m+2\epsilon,$ and so $d_{S}(y_{i}',y_{i+1}') \leq \lambda (m+2\epsilon)+C \leq \lambda'm$ for some fixed $\lambda'$ as $m \geq r$. Note that we can choose $y_{0}'$ to be $s$, and $y_{K}'$ to be $y$. Thus $\nu_{S}(\lambda'm,n) \leq \nu_{T}(m,\lambda n+C).$ If $\lambda n+C \leq m$, we have that $\nu_{T}(m,\lambda n+C) = 1$, so we can assume that $\lambda n+C$ is greater than $r$ as well, and we have that $\nu_{S}(\lambda'm,n) \leq \nu_{T}(m,\lambda'' n)$ for some fixed $\lambda''$. By symmetry, $\nu_{T} \preceq \nu_{S}$, and so $\nu_{T} \cong \nu_{S}$.
\end{proof}

\begin{defn}
Call a metric space $S$ $\textit{asymptotically transitive}$ if $Cone^{\omega}(S)$ is transitive for all ultrafilters $\omega$.
\end{defn}

\begin{thm} 
   Let $r$ be a positive number and let $(S,s)$ be an asymptotically transitive proper r-connected pointed metric space. The following are equivalent:
\begin{enumerate}

\item there exists a function $f \colon \mathbb{R}^{\geq 0} \rightarrow \mathbb{R}^{\geq 0}$ such that for all $m \geq r, n \geq 0$, \newline $\nu_{S}(m,n) \leq f(n/m)$,

\item there exists a constant $K$ such that $\nu_{S}(i,4i) \leq K$ for all real numbers $i \geq r$,

\item $Cone^{\omega}(S)$ is path connected for all non-principal ultrafilters $\omega$,

\item $Cone^{\omega}(S)$ is connected for all non-principal ultrafilters $\omega$.
\end{enumerate}
\end{thm}
Note that the implication $1) \Rightarrow 2)$ is clear, simply by letting $K = f(4)$. The implication $3) \Rightarrow 4)$ is also immediate.

To show that $2)$ implies $3)$ we will need the following lemma.

\begin{lem}
    Let $r \in \mathbb{R}^{\geq 0}$. If $(S,s)$ is an asymptotically transitive, proper, $r$-connected, pointed metric space and there exists a constant $K$ such that $\nu_{S}(i,4i) \leq K$ for all real numbers $i \geq r$, then for any points $p=(y_{i})^{\omega}, q=(z_{i})^{\omega} \in Cone^{\omega}(S)$, there exist $K+1$ points \newline $p= p_{0}, \, p_{1}, \, p_{2}, \,...,\, p_{K} = q$ in $Cone^{\omega}(S)$ such that $d^{\omega}_{S}(p_{i},p_{i+1}) \leq d_{S}^{\omega}(p,q)/2$.
\end{lem}

\begin{proof} 
    If $(y_{i})^{\omega} = (z_{i})^{\omega}$, the result is trivial, so let $(y_{i})^{\omega}$ and $(z_{i})^{\omega}$ be points in  $Cone^{\omega}(S)$ such that $d_{S}^{\omega}((y_{i})^{\omega},(z_{i})^{\omega}) = C > 0$. Note that by the transitivity of $Cone^{\omega}(S)$, we can assume that $(y_{i})^{\omega} = (s)^{\omega}$. This means in particular that $d_{S}(s,z_{i}) \leq 2Ci$ $\omega$-almost surely. Note that $Ci/2 \geq r$ $\omega$-almost surely, and hence $\nu_{S}\left(Ci/2, 2Ci\right) \leq K \textrm{ }\omega \textrm{-almost surely}.$ It follows that there exist points $s=y_{i,0}, \, y_{i,1}, \,..., \, y_{i,K} = z_{i}$ with $d_{S}(y_{i,j},y_{i,j+1}) \leq Ci/2$ for all $0 \leq j \leq K-1$ $\omega$-almost surely.
   Now define $p_{j} = (y_{i,j})^{\omega}.$ Note that $d_{S}^{\omega}(p_{j},p_{j+1}) = \textrm{lim}^{\omega} d_{S}(y_{i,j},y_{i,j+1})/i \leq C/2,$ and so we have our desired $p_{0}, \,..., \, p_{K}$. 
\end{proof}

We will also need the following Lemma in order to prove that $4)$ implies $1)$.

\begin{lem} 
    If $S$ is a connected metric space, then for any real number $r > 0$, $S$ is r-connected.
\end{lem}

\begin{proof} 
For a fixed $r > 0$, and fixed $p \in S$, consider the set $C$ of points $q$ such that there exists a finite sequence of points $p = p_{0}, \, p_{1} \, \dots \, p_{K} = q$ with $d(p_{i},p_{i+1}) \leq r$. If $x \in C$, then clearly $B_{r}(x) \subset C$, and so $C$ is open. Similarly, if $x \not \in C$, then $B_{r}(x) \subset S \setminus C$, so $C$ is closed. Hence, $C$ is open, closed and non-empty, so $C = S$, as desired.
\end{proof}

We are now ready to prove the theorem.

\begin{proof} 
	We begin by proving $2)$ implies $3)$.
	
    Let $p,q \in Cone^{\omega}(S)$, and let $C = d^{\omega}_{S}(p,q)$. We will define a uniformly continuous function $f$ from numbers of the form $a/K^{n}$ with $a,n \in \mathbb{N}$ $a \leq K^{n}$ to the asymptotic cone such that 
$f(0)=p$ and $f(1) = q$. Note that this is sufficient, since asymptotic cones are complete, and these numbers are dense in the interval $[0,1]$.

    We will define the function inductively as follows. First, define $f(0) = p$ and $f(1) = q$. Then, fix $n \in \mathbb{N}$, and assume we've defined $f$ on all numbers of the form $a/K^{n}$ in such  a way that for all $s \in \mathbb{N} \cup \{0\}$ with $s < K^{n}$ $$d^{\omega}_{S}\left(f\left(\frac{s}{K^{n}}\right),f\left(\frac{s+1}{K^{n}}\right)\right) \leq \frac{C}{2^{n}}.$$ Now let $t = (K\ell+b)/K^{n+1}$ where $1 \leq b < K$ and $\ell \in \mathbb{N} \cup \{0\}$, $\ell \leq K^{n-1}$ According to Lemma 4.8, there exist points $p_{0},p_{1},\dots,p_{K}$ such that $$f\left(\frac{\ell}{K^{n}}\right) = p_{0},p_{1},...,p_{K} = f\left(\frac{\ell+1}{K^{n}}\right),$$ and $$d^{\omega}_{S}(p_{i},p_{i+1}) \leq \frac{d_{S}^{\omega}(f(\frac{\ell}{K^{n}}),f(\frac{\ell+1}{K^{n}}))}{2} \leq \frac{C}{2^{n+1}}.$$ Let $f(t) = p_{b}$. It is straightforward to verify that $f$ is uniformly continuous.

    We will now show that $4)$ implies $1)$ by contradiction. Assume that $Cone^{\omega}(S)$ is connected, and that $\nu_{S}(m,n)$ is not bounded by any homogeneous function. Hence there exists a $c \in \mathbb{R}^{>0}$ such that $\nu_{S}(n,cn)$ is not bounded. Let $n_{i}$ be a sequence of natural numbers such that $\nu_{S}(n_{i},cn_{i}) \geq i$. Let $\omega$ be an ultrafilter containing $\{n_{i} | i \in \mathbb{N}\}$. Consider a sequence of points $t_{i} \in S$ such that $d_{S}(s,t_{i}) \leq ci$, and $|t_{i}|_{i} = \nu_{S}(i,ci)$. According to Lemma 4.9, we can pick points $(s)^{\omega} = p_{0},p_{1},...,p_{k} = (t_{i})^{\omega}$ in $Cone^{\omega}(S)$ such that $d^{\omega}_{S}(p_{i},p_{i+1}) \leq \frac{1}{2}$. Let $p_{j} = (t_{i,j})^{\omega}$. We have that $d_{S}(t_{i,j},t_{i,j+1}) \leq i \textrm{ } \omega\textrm{-almost surely},$ so $\nu_{S}(i,ci)=|t_{i}|_{i} \leq k \textrm{ }\omega\textrm{-almost surely}.$ On the other hand if $j > k$, then $\nu_{S}(n_{j},cn_{j}) > k$. However, $$\{n_{j}| j > k\} = \{n_{j} | j \in \mathbb{N}\} \cap \{n | n > n_{k}\} \in \omega,$$ a contradiction. 
\end{proof}

We now want to study how distortion of groups relates to connectedness in asymptotic cones. We begin by defining a natural subspace of the asymptotic cone of $G$ corresponding to $H$.

\begin{defn} 
    Let $T$ be a subspace of a metric space $S$. Denote by $Cone^{\omega}_{S}(T)$ the set of all points in $Cone^{\omega}(S)$ with a representative $(t_{i})^{\omega}$ with each component in $T$. 
\end{defn}

\begin{lem} 
    For all subspaces $T \subset S$, $\textrm{Cone}^{\omega}_{S}(T)$ is closed in $\textrm{Cone}^{\omega}(S)$.
\end{lem}

\begin{proof}
    Note that $Cone^{\omega}_{S}(T) = Cone^{\omega}(T)$ where we consider $T$ under the induced metric from $S$. Since asymptotic cones are complete, this is a complete metric space. A complete subspace of a complete metric space is closed and so we have that $Cone^{\omega}_{S}(T)$ is closed in $Cone^{\omega}(S)$.
\end{proof}

Note that we can think about a subgroup $H$ of a group $G$ as a subspace of the metric space we get by considering the word metric on $G$. 

\begin{lem}
If $H$ is a subgroup of a finitely generated group $G$ such that $Cone^{\omega}_{G}(H)$ is connected for all ultrafilters $\omega$, then $H$ is finitely generated.
\end{lem}

\begin{proof}
Let $H$ be a subgroup of a finitely generated group $G$, and let $X$ be a finite generating set for $G$. We call an element $h$ of $H$ $\textit{reducible}$ if there exists a constant $k \in \mathbb{N}$ and $k$ elements of $H$, $h_{1},\, h_{2} \, \dots \, h_{k}$, with $|h_{i}|_{X} < |h|_{X}$ for all $0 \leq i \leq k$ such that $h = h_{1}h_{2}\dots h_{k}$. We call an element $h \in H$ $\textit{irreducible}$ if it is not reducible. We can assume that there exists no $i$ such that all elements $h \in H$ with $|h|_{X} \geq i$ are reducible, as this would imply that $H$ is finitely generated. Thus we can find a sequence $(h_{i})$ of irreducible elements of $H$ such that $|h_{i}|_{X} > |h_{i-1}|_{X}$ for all $i$. Fix an ultrafilter $\omega$ and consider the asymptotic cone $Cone^{\omega}_{G}(H)$ with respect to $\omega$ and the scaling sequence $(|h_{i}|_{X})$. Assume this asymptotic cone is connected. As $(h_{i})^{\omega} \in Cone^{\omega}_{G}(H)$, there exist points $(e)^{\omega} = p_{0}, \, p_{1}, \, \dots \, , p_{k} = (h_{i})^{\omega}$ with $d(p_{i},p_{i+1}) \leq 1/4$ for all $0 \leq i < k$. Let $p_{j} = (h_{i,j})^{\omega}$. We have that $|h_{i,j}^{-1}h_{i,j+1}|_{X} \leq |h_{i}|_{X}/2 \textrm{ }\omega\textrm{-almost surely}.$  Finally, note that $h_{i} = h_{i,k} = h_{1,i}(h_{i,1}^{-1}h_{i,2})\dots(h_{i,k-1}^{-1}h_{i,k})$. This, however, implies that $h_{i}$ is $\omega$-almost surely reducible, a contradiction.
\end{proof}

We can apply Theorem 4.8 to a subgroup $H$ of a finitely generated group $G$, where $H$ is given the word metric induced from $G$. In this case, the relationship between $\nu_{H}$ and $\mu^{G}_{H}$ combined with theorem 4.14 gives the following theorem.
\begin{thm} 
    The following are equivalent for a  subgroup $H$ of a finitely generated group $G$:
\begin{enumerate}

\item H is finitely generated and there exists a constant $K$ such that $\mu_{H}^{G}(i,4i) \leq K$ for all $i$.

\item H is finitely generated and there exists a function $f$ such that $\mu_{H}^{G}(m,n) \leq f(\frac{n}{m})$.

\item $Cone_{G}^{\omega}(H)$ is path connected for all ultrafilters $\omega$.

\item $Cone_{G}^{\omega}(H)$ is connected for all ultrafilters $\omega$.
\end{enumerate}
\end{thm}

\begin{example}
We have previously seen that if $G = BS(1,p) = \langle a,b \mid b^{-1}ab = a^{p}\rangle$, and $H = \langle a \rangle$  then $\mu^{G}_{H}(m,n) \cong  p^{n-m}$. Thus $\mu^{G}_{H}(i,2i)$ is unbounded, and there exists an ultrafilter $\omega$ such that $Cone^{\omega}_{G}(H)$ is disconnected.
\end{example}
\begin{example}
If $G$ is the discrete Heisenberg group, and $H$ is the center of $G$, then we have seen in a previous example that $\mu^{G}_{H}(m,n) \cong  n^{2}/m^{2},$ and so $\mu^{G}_{H}(i,4i)$ is bounded, and $Cone^{\omega}_{G}(H)$ is connected for all ultrafilters $\omega$.
\end{example}

We now want to relate the connectedness of $Cone^{\omega}_{G}(H)$ to the distortion of $H$ in $G$. In order to do this, we need a couple preliminary results. The first of these is due to Olshanskii.

\begin{thm}\cite{olshanskii}
For any group $H$, and any function $\ell \colon H \rightarrow \mathbb{N}$ satisfying the following conditions:
\begin{enumerate}
    \item for all $h \in H$, $\ell(h)=0$ if and only if $h = 1$,
    \item $\ell(h) = \ell(h^{-1})$ for all $h \in H$,
    \item $\ell(gh) \leq \ell(g) + \ell(h)$ for all $g,h \in H$,
    \item there exists a constant $a$ such that $|\{h \in H \mid \ell(h) \leq n\}| \leq a^{n}$,
    \end{enumerate}
there exists a group $G = \langle X \rangle$ with $|X| < \infty$ , an embedding $\phi$ of $H$ in $G$, and a constant $C$ such that for all $h \in H$, $$\frac{|\phi(h)|_{X}}{C} \leq \ell(h) \leq C|\phi(h)|_{X}.$$
\end{thm}

\begin{defn} A function $f \colon \mathbb{R}^{\geq 1} \rightarrow \mathbb{R}$ is called \textit{superlinear} if for all $k \in \mathbb{R}$ the set $\{n \mid f(x) \leq kx\}$ is bounded. $f$ is called \textit{sublinear} if for all $k \in \mathbb{R}$ the set $\{x \mid f(x) \geq kx\}$ is bounded.
\end{defn}

\begin{lem}
Let $f \colon \mathbb{R}^{\geq 1} \rightarrow \mathbb{R}$ be an increasing, sublinear function with $f(r) \leq r$ for all real numbers $r \geq 1$. There exists a function $\ell \colon \mathbb{R}^{\geq 1} \rightarrow \mathbb{R}^{\geq 1}$ satisfying the following properties:
\begin{enumerate}
\item for all $m,n \in \mathbb{N}$, $\ell (m) + \ell(n) \geq \ell (m+n)$. 
\item for all $n \in \mathbb{N}$, $\ell(n) \geq f(n)$.
\item for all $k \in \mathbb{N}$, there exists a $p_{k} \in \mathbb{N}$ such that $\ell(p_{k}) = \ell(p_{k+1}) = \dots = \ell(kp_{k}).$
\end{enumerate}
\end{lem}
\begin{proof}
We will define $p_{k}$ and $\ell$ by induction on $k$. First let $p_{1} = 1$ and let $\ell(1) = 1$. Assume we have defined $p_{k}$ and $\ell(n)$ for $n \leq kp_{k}$ in a way that satisfies properties 1-3. Let $p_{k+1}$ be the least real number such that for all $r \in \mathbb{R}$, if $r \geq (k+1)p_{k+1}$, then $f(r) \leq r/(k+1)!$. For $s \in \mathbb{R}$, if $kp_{k} < s \leq p_{k+1}$ define $\ell(s) = s/k!$. For $s \in \mathbb{R}$, $p_{k+1} \leq s \leq (k+1)p_{k+1}$, define $\ell(s) = p_{k+1}/k!$. By definition, $\ell((k+1)p_{k+1}) = p_{k+1}/k! = (k+1)p_{k+1}/(k+1)!$.

We will now show that $\ell$ satisfies properties 1-3. First, fix $r \in \mathbb{R}^{\geq 1}$, and let $k \in \mathbb{N}$ such that $kp_{k} \leq r \leq (k+1)p_{k+1}$. If $kp_{k} < r < p_{k+1}$, then $\ell(r) = r/k!$, and if $s < r$, then $\ell(s) \geq s/k!$. Thus, if $p + q = r$, then $\ell(p) + \ell(q) \geq p/k! + q/k! = r/k! = \ell(r)$. If $p_{k+1} < r \leq (k+1)p_{k+1}$, then $\ell(r) = \ell(p_{k+1})$, and property 1 follows immediately as $\ell$ is increasing. For $s \in \mathbb{R}$, if $kp_{k} \leq s \leq p_{k+1}$, then $\ell(s) = s/k! > f(s)$ by definition. If $p_{k+1} \leq s \leq (k+1)p_{k+1}$, then $\ell(s) =\ell((k+1)p_{k+1}) = (k+1)p_{k+1}/(k+1)! \geq f((k+1)p_{k+1}) \geq f(s)$, so $\ell$ satisfies property 2. It is clear that this definition of $\ell$ satisfies property 3. 
\end{proof}

We are now ready to relate the connectedness of $Cone^{\omega}_{G}(H)$ to the distortion of $H$ in $G$.

\begin{thm}
\begin{enumerate} If $H$ is a finitely generated subgroup of a finitely generated group $G$, then the following implications hold.
    \item If $\Delta^{G}_{H}(n)$ is linear, then $Cone^{\omega}_{G}(H)$ is connected for all ultrafilters $\omega$.
    \item If $Cone^{\omega}_{G}(H)$ is connected for all ultrafilters $\omega$, then $\Delta^{G}_{H}(n) \preceq f$ for some polynomial $f$.
    \item For every increasing, superlinear function $\phi \colon \mathbb{N} \rightarrow \mathbb{N}$ there exists a group $G$ with a subgroup $H$ such that $Cone^{\omega}_{G}(H)$ is disconnected for some ultrafilter $\omega$, but $\Delta^{G}_{H}(n) \preceq \phi$.
    \item For all $k \in \mathbb{N}$, there exists a group $G$ with a subgroup $H$ such that $Cone^{\omega}_{G}(H)$ is connected for all ultrafilters $\omega$, and $\Delta^{G}_{H} \sim n^{k}$.
\end{enumerate}
\end{thm}

\begin{proof}
We will begin by proving claim 1.

If $H$ is a subgroup of $G$, then we can define a continuous function $\rho$ from $Cone^{\omega}(H)$ to $Cone^{\omega}_{G}(H)$ by $\rho((h_{i})^{\omega}) = (h_{i})^{\omega}$. For all $h \in H$, $|h|_{X} \leq C|h|_{Y}$ for some fixed constant $C$, so $\rho$ is well-defined. Assume $(h_{i})^{\omega} \in Cone^{\omega}_{G}(H)$. This means that there exists  $B$ such that for all $i \in \mathbb{N}$, $|h_{i}|_{X}/i \leq B.$ Distortion is linear means that there exists $D$ such that $ \frac{|h_{i}|_{Y}}{i} \leq D\frac{|h_{i}|_{X}}{i} \leq DB.$ Thus, $\rho$ is surjective, and $Cone^{\omega}_{G}(H)$ is connected, as $Cone^{\omega}_{G}(H)$ is connected.

Now we prove the second claim in Theorem 4.17.

Assume that $Cone^{\omega}_{G}(H)$ is connected in $Cone^{\omega}(G)$, and hence that $\mu^{G}_{H}(i,2i)$ is bounded by some constant $K$ for all $i$. By induction we have that $\Delta^{G}_{H}(2^{n}) = \mu_{H}^{G}(1,2^{n}) \leq K^{n}$ for all $n \in \mathbb{N}$. 

Now let $n \in \mathbb{N}$, and let $m \in \mathbb{R}$ such that $2^{m-1} \leq n < 2^{m}$. We have that $$\Delta^{G}_{H}(n) \leq \Delta^{G}_{H}(2^{m}) \leq K^{m} =(2^{m})^{\log_{2}K} \leq (2n)^{\log_{2}K}.$$ Thus, $\Delta^{G}_{H}(n) \preceq n^{\log_{2}K}.$

We will now prove the third claim of the theorem. Let $\phi$ be a superlinear, increasing function $\mathbb{N} \rightarrow \mathbb{N}$. $\phi$ can be extended to an invertible, increasing, superlinear function from $\mathbb{R}^{\geq 1}$ to $\mathbb{R}$. We can now apply Lemma 4.17 to $\phi^{-1}$ to get a function $\ell$ which is always larger than $\phi^{-1}$. We can then restrict $\ell$ to the natural numbers and take ceilings to get a function from $\mathbb{N}$ to $\mathbb{N}$. We can extend this to a function from $\mathbb{Z}$ to $\mathbb{Z}$ by defining $\ell(0) = 0$ and $\ell(-z) = \ell(z)$ for $z < 0$. As $\ell \geq \phi^{-1}$, we have that $\phi(\ell(n)) \geq n$.
If $\phi$ is subexponential, then this $\ell$ now satisfies all of the conditions of Theorem 4.16, and hence there exists a group $G = \langle X \rangle$, a constant $C$ and an embedding $\psi \colon \mathbb{Z} \rightarrow G$ such that $$ \frac{\ell(n)}{C} \leq |\psi(n)|_{X} \leq C\ell(n).$$ Now note that if $|\psi(n)|_{X} \leq m$, then $\ell(n) \leq C|\psi(n)|_{X} \leq Cm$, and so $n < \phi(\ell(n)) \leq \phi(Cm).$ Hence, distortion is bounded by $\phi$. On the other hand, $\ell(p_{k})=\ell(p_{k}+1)=\dots=\ell(kp_{k})$ implies that $C|\psi(q)|_{X} > \ell(p_{k})$ for all $p_{k} 
\leq q \leq kp_{k}$ while $|\psi(kp_{k})|_{X} \leq C\ell(p_{k})$, and so $\mu^{G}_{H}\left(\ell(p_{k})/C,C\ell(p_{k})\right) \geq k.$ By Theorem 4.15, $Cone^{\omega}_{G}(H)$ is disconnected for some ultrafilter $\omega$. 

Note that if $\phi$ is superexponential, then claim 2 of Theorem 4.19 shows that $Cone^{\omega}_{G}(H)$ is not connected for all ultrafilters $\omega$. 

Part 4 of the theorem can also be proven using this method.

Fix $k \in \mathbb{N}$, and for $z \in \mathbb{Z}$ let $\ell(z) = \left \lceil{|z|}^{\frac{1}{k}}\right \rceil$. Let $G$ be a group with finite generating set $X$ and $\psi$ an embedding of $\mathbb{Z}$ into $G$ such that $$\frac{\ell(z)}{C} \leq |\psi(z)|_{X} \leq C\ell(z).$$ Note that if $|\psi(z)|_{X} \leq m$, then $|z|^{1/k} \leq \left \lceil{|z|}^{1/k}\right \rceil = \ell(z) \leq C|\psi(z)|_{X} \leq Cm,$  which implies that $|z| \leq C^{k}m^{k}.$  Thus $\Delta^{G}_{H}(m) \preceq m^{k}$. Now note that $\ell(m^{k}) = m$, so $|\psi(m^{k})|_{X} \leq Cm$, which implies $\Delta^{G}_{H}(Cm) \geq m^{k}.$ Thus, $\Delta^{G}_{H}(m) \sim m^{k}.$ The above calculations show  that if $|\psi(z)|_{X} \leq 4i$, then $|z| \leq 4^{k}C^{K}i^{k}$. Further, if $|z| \leq (i/C)^{K}$ then $|\psi(z)|_{X} \leq C\ell(z) \leq i.$ Thus, $\mu^{G}_{H}(i,4i) \leq 4^{k}C^{2k}$, and so by Theorem 4.20 we have that $Cone^{\omega}_{G}(H)$ is connected.
\end{proof}

\section{Convexity in asymptotic cones}

\begin{defn}
A subspace $T$ of a metric space $S$ is called $\textit{Morse}$ if for all constants $\lambda, C$ there exists a constant $M$ such that any $(\lambda,C)$-quasi-geodesic connecting points in $T$ is contained in the $M$ neighborhood of $T$.
\end{defn}

\begin{defn}
We say a subset $T$ of a metric space $S$ is $\textit{strongly}$ $\textit{convex}$ if every simple path starting and ending in $T$ is entirely contained in $T$.
\end{defn}

\begin{thm}
Let $T$ be a closed subspace of a geodesic metric space $S$. Assume that $Cone^{\omega}_{S}(T)$ is strongly convex in $Cone^{\omega}(S)$ for all ultrafilters $\omega$ and for any two points $t_{1},t_{2}$ in $Cone^{\omega}_{S}(T)$ there exists an isometry $\phi$ of $Cone^{\omega}(S)$ fixing $Cone^{\omega}_{S}(T)$ such that $\phi(t_{1}) = t_{2}$. Then $T$ is Morse.
\end{thm}

\begin{proof}
Assume $T$ is not Morse. This means that there exist constants $\lambda \geq 1, C \geq 0$ such that for all $i \in \mathbb{N}$ there exists a $(\lambda,C)$-quasi-geodesic $p_{i}\colon[0,k_{i}] \rightarrow S$ parameterized by length, and $s_{i} \in [0,k_{i}]$ with $p_{i}(0)$ and $p_{i}(k_{i})$ in $T$ and $d_{S}(p_{i}(s_{i}),T) \geq i$. For all $i$ let \begin{equation} \tag{2} \label{eqn:a} d_{i} = \sup\{d_{S}(p_{i}(s),T)\mid s \in [0,k_{i}]\}. 
\end{equation} 
We can choose our paths $p_{i}$ to make the sequence $(d_{i})$ increasing with all $d_{i} > C$. For each $i$, let $s_{i}$ be a point in $[0,k_{i}]$ such that $d_{S}(p_{i}(s_{i}),T) = d_{i}$ (such a point exists as paths are compact). Let $s_{i}^{\ell} = \textrm{max}\{s_{i}-3\lambda d_{i}, 0\},$ and similarly let $s_{i}^{r} = \textrm{min}\{s_{i} + 3 \lambda d_{i}, k_{i} \}.$
\begin{figure}[h]
\centering
\includegraphics[scale=.7]{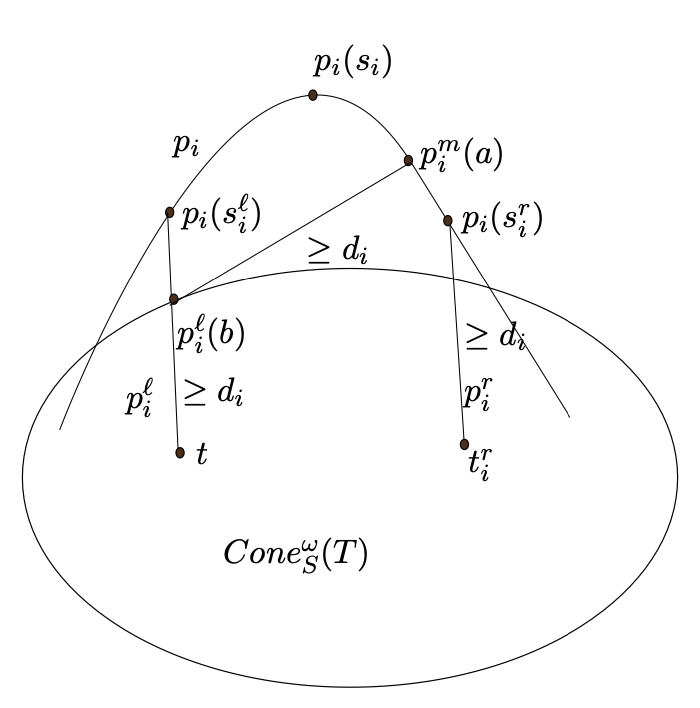}
\caption{Theorem 5.3}
\end{figure}
By \eqref{eqn:a} $d_{S}(p_{i}(s_{i}^{\ell}),T)$ and $d_{S}(p_{i}(s_{i}^{r}),T)$ are less than or equal to $d_{i}$. Let $d_{S}(p_{i}(s_{i}^{\ell}),T) = k_{i}^{\ell},$ and $d_{S}(p_{i}(s_{i}^{r}),T) = k_{i}^{r}.$ Let $t_{i}^{\ell}$ be a point in $T$ such that $d_{S}(p_{i}(s_{i}^{\ell}),t_{i}^{\ell}) = k_{i}^{\ell},$ and let $p_{i}^{\ell}\colon[0,k_{i}^{\ell}]\rightarrow \Gamma(G)$ be a geodesic from $t_{i}^{\ell}$ to $s_{i}^{\ell}$. Note that by assumption we can take $t_{i}^{\ell} = t$ where $t$ is some fixed point in $T$ by taking an isometry fixing $T$ sending $t_{i}^{\ell}$ to $t$. Similarly, let $p_{i}^{r}\colon [0,k_{i}^{r}]$ be a geodesic from $s_{i}^{r}$ to a point $t_{i}^{r} \in T$ such that $d_{S}(t_{i}^{r}, p_{i}(s_{i}^{r})) = k_{i}^{r}$. Denote by $p_{i}^{m}\colon [s_{i}^{\ell},s_{i}^{r}] \rightarrow S$ the segment of $p_{i}$ from $p_{i}(s_{i}^{\ell})$ to $p_{i}(s_{i}^{r})$. 

We will need the following lemma.
\begin{lem} \begin{enumerate}
\item For all $i \in \mathbb{N}$, if $s_{i}^{\ell} \neq 0$, $a \in [s_{i},s_{i}^{r}]$, and $b \in [0,k_{i}^{\ell}]$, then \newline $d_{S}(p_{i}^{m}(a),p_{i}^{\ell}(b)) \geq d_{i}$. 
\item For all $i \in 
\mathbb{N}$, if $s_{i}^{r} \neq k_{i}$, $a \in [s_{i}^{\ell},s_{i}]$, and $b \in [0,k_{i}^{r}]$, then  $d_{S}(p_{i}^{m}(a),p_{i}^{r}(b)) \geq d_{i}$.
\end{enumerate}
\end{lem}
\begin{proof} First, if $s_{i}^{\ell} \neq 0$, then $s_{i}^{\ell} = s_{i}-3\lambda d_{i}$. Now note that $$d_{S}(p_{i}^{m}(a),p_{i}^{m}(s_{i}^{\ell})) \geq \frac{3\lambda d_{i}}{\lambda}-C = 3d_{i}-C > 3d_{i} - d_{i} = 2d_{i},$$ as $p_{i}$ is a $(\lambda,C)$ geodesic, and we assumed that $d_{i} > C$. Thus, as $d_{S}(p_{i}^{\ell}(b),p_{i}^{m}(x_{i}^{\ell})) \leq d_{i}$, $d_{S}(p_{i}^{m}(a),p_{i}^{\ell}(b)) \geq d_{i}$. The second claim follows similarly.
\end{proof}

We return to the proof of Theorem 5.3.

Fix an ultrafilter $\omega$, and consider the asymptotic cone of $S$ with respect to $\omega$ and the scaling sequence $d_{i}$. By construction, $d_{S}(t,p_{i}^{\ell}(k_{i}^{\ell})) \leq d_{i},$ and so $(p_{i}^{\ell}(k_{i}^{\ell}))^{\omega} \in Cone^{\omega}(G).$ As $|s_{i}^{\ell}-s_{i}^{r}| \leq 6\lambda d_{i},$ we have that $d_{S}(p_{i}(s_{i}^{\ell}),p_{i}(s_{i}^{r})) \leq 6\lambda^{2}d_{i}+C,$. and so as $(p_{i}(s_{i}^{\ell}))^{\omega} \in Cone^{\omega}(G)$, we have that $(p_{i}(s_{i}^{r}))^{\omega} \in Cone^{\omega}(G)$. As $d_{S}(p_{i}(s_{i}^{r}),p_{i}^{r}(k_{i}^{r})) = d(p_{i}^{r}(0),p_{i}^{r}(k_{i}^{r})) \leq d_{i},$ we have that $(p_{i}^{r}(k_{i}^{r}))^{\omega} \in Cone^{\omega}(G)$. Thus we can define $$k^{\ell} = \textrm{lim}^{\omega} \frac{k^{\ell}_{i}}{d_{i}}, s^{\ell} = \textrm{lim}^{\omega} \frac{s^{\ell}_{i}}{d_{i}}, s^{r} = \textrm{lim}^{\omega} \frac{s_{i}^{r}}{d_{i}}, k^{r} = \textrm{lim}^{\omega}\frac{k_{i}^{r}}{d_{i}},$$ and we can define $p^{\ell}\colon[0,k^{\ell}] \rightarrow Cone^{\omega}(S)$ as $\lim^{\omega}(p_{i}^{\ell})$, $p^{m} \colon [s^{\ell},s^{r}] \rightarrow Cone^{\omega}(S)$ as $\lim^{\omega}(p_{i}^{m})$, and $p^{r} \colon [0,k^{r}]$ as $\lim^{\omega}(p_{i}^{r})$. We have that $p^{\ell}$ and $p^{r}$ are geodesics, and $p^{m}$ is a $(\lambda,0)$ quasi-geodesic, and hence all are simple. 

Now we have three simple paths, $p^{\ell},p^{m},p^{r}$, such that $p^{\ell}(0)$ and $p^{r}(k^{r})$ are in $Cone^{\omega}_{S}(T)$, and $p^{\ell}$ and $p^{r}$ both intersect $p^{m}$. Unfortunately, the concatenation of these three paths may not be simple, as $p^{\ell}$ and $p^{r}$ could intersect $p^{m}$ in more than once. To deal with this case, we need the following lemma.

\begin{lem} 
\begin{enumerate} Let $s = \textrm{lim}^{\omega} s_{i}/d_{i}.$
\item If $a \in [0,k^{\ell}]$,and $b \in [s^{\ell},s^{r}]$, with $p^{\ell}(a) = p^{m}(b)$, then $b \leq s$.
\item if $a \in [0,k^{r}]$, and $b \in [s^{\ell},s^{r}]$, with $p^{r}(a) = p^{m}(b),$ then $b \geq s$.
\end{enumerate}
\end{lem}
\begin{proof}
Note that if $\{i|k_{i}^{\ell} = 0\} \in \omega,$ then $p^{l}$ is a trivial path, and the result is clear. Otherwise, we have that $\{i| k_{i}^{\ell} \neq 0\} \in \omega.$ In this case we can use Lemma 5.4 to say that if $(b_{i})^{\omega}$ is on $p^{\ell}$ and $(a_{i})^{\omega}$ is on $p^{m}$ after $s$, then $d_{S}^{\omega}((b_{i})^{\omega},(a_{i})^{\omega}) \geq \textrm{lim}^{\omega}\frac{d_{i}}{d_{i}} \geq 1.$ The proof of claim 2 follows similarly.
\end{proof}

Thus, we can form a simple path which starts and ends in $Cone^{\omega}_{S}(T)$ as follows. Let $$p = \textrm{max}\{t \in [s^{\ell},s^{r}]\mid \exists a \in [0,k^{\ell}] \, \, p^{\ell}(a) = p^{m}(t)\},$$ and let $$q = \textrm{min}\{t \in [s^{\ell},s^{r}]\mid \exists a \in [0,k^{r}]\, \, p^{r}(a) = p^{m}(t)\}.$$ We obtain a simple path by following $p^{\ell}$ up to $p^{m}(p)$, then following $p^{m}$ up to $p^{m}(q)$, and finally following $p^{r}$ back to $p^{r}(k^{r})$. This path contains $p^{m}(s)$ by Lemma 5.5. Finally, as $p^{m}(s) = (p_{i}^{m}(s_{i}))^{\omega},$ $$d_{S}^{\omega}(p^{m}(s),Cone^{\omega}_{S}(T)) = \textrm{lim}^{\omega}\frac{d_{S}(p_{i}^{m}(s_{i}),Cone^{\omega}_{S}(T))}{d_{i}} = \textrm{lim}^{\omega}\frac{d_{i}}{d_{i}}=1.$$ 

Thus, we have a simple path starting and ending in $Cone^{\omega}_{S}(T)$, which is not entirely contained in $Cone^{\omega}_{S}(T)$.
\end{proof}

In order to prove a partial converse of this statement we will need the following results from Drutu, Mozes and Sapir \cite{Sapir}. Note that an error was found in this paper \cite{errata}, but none of the following lemmas were affected.

\begin{lem} (\cite{Sapir} Lemma 2.3) Let $S$ be a geodesic metric space, $\omega$ an ultrafilter, and $B$ a closed subset of $Cone^{\omega}(S)$. If $x,y$ are in the same connected component of $Cone^{\omega}(S)\setminus B$, then there exists a sequence of paths $(p_{i})_{i=1}^{n}$ such that each path is a limit geodesic in $X$, and the concatenation of the paths $p_{i}$ is a simple path from $x$ to $y$. \end{lem}

\begin{defn}
A path is called $C$ bi-lipschitz if it is a $(C,0)$ quasi-geodesic.
\end{defn}

\begin{lem} (\cite{Sapir} Lemma 2.5)
In the same setting as Lemma 5.6, let $p$ be a simple path in $Cone^{\omega}(S)$ which is a concatenation of limit geodesics. For all $\delta$ there exists a constant $C$ and a $C$ bi-Lipschitz path $p'$ such that the Hausdorff distance between $p$ and $p'$ is less than $\delta$, and $p'$ is also a concatenation of limit geodesics connecting the same points.
\end{lem}

\begin{lem} (\cite{Sapir} Lemma 2.6)
Let $p$ be a $C$-bi-Lipschitz path in $Cone^{\omega}(S)$ which is a concatenation of limit geodesics. There exists a constant $C'$ and a sequence of paths $(p_{n})$ in $S$ such that each $p_{n}$ is $C'$ bi-Lipschitz, and $lim^{\omega}(p_{n}) = p$.
\end{lem}

\begin{thm}
If $T$ is a Morse subspace of a metric space $S$, then $Cone^{\omega}_{S}(T)$ is strongly convex in $Cone^{\omega}(S)$.
\end{thm}

\begin{figure}[h]
\centering
\includegraphics[scale=.6]{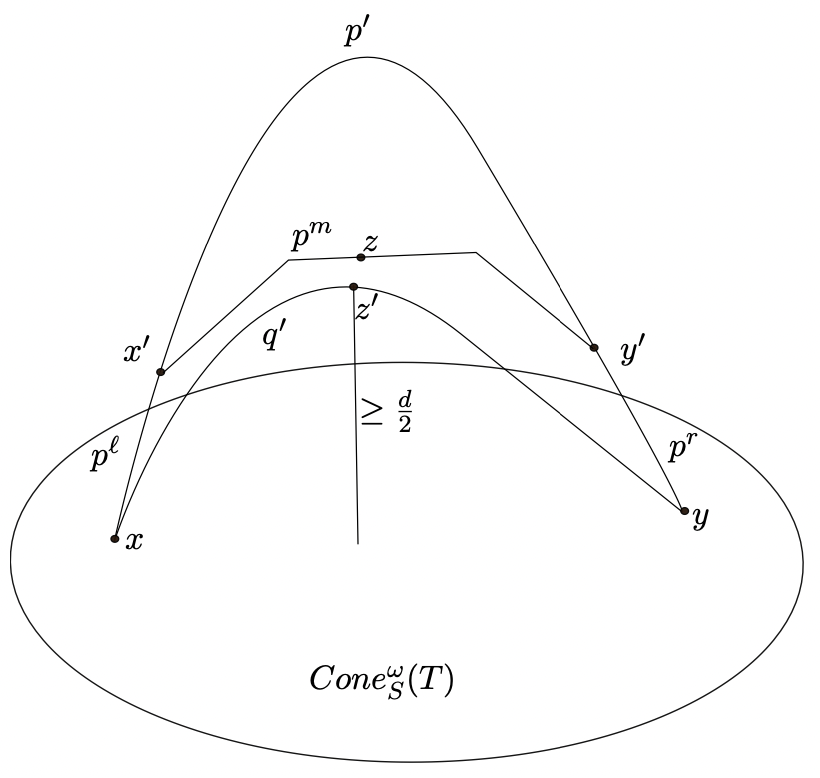}
\caption{Theorem 5.9}
\end{figure}

\begin{proof}
Let $p$ be a simple path in $Cone^{\omega}(S)$ starting and ending in $Cone^{\omega}_{S}(T)$ but not entirely contained in $Cone^{\omega}_{S}(T)$. As $Cone^{\omega}_{S}(T)$ is closed, there is a subpath $p'$ of $p$ which starts and ends in $Cone^{\omega}_{S}(T)$ but no interior point of $p'$ is in $Cone^{\omega}_{S}(T)$. Let $x$ be the initial point of $p$ and let $y$ be the terminal point of $p$. Let $x',y'$ be points on $p'$ such that $$\textrm{max}\{d_{S}^{\omega}(x,x'),d_{S}^{\omega}(y,y')\} < \frac{d_{S}^{\omega}(x,y)}{2},$$ and let $p^{l},p^{r}$ be limit geodesics from $x$ to $x'$ and from $y'$ to $y$ respectively. Let $p^{m}$ be a concatenation of limit geodesics connecting $x'$ to $y'$ avoiding $Cone^{\omega}_{S}(T)$. Such a path exists by Lemma 5.6 as $Cone^{\omega}_{S}(T)$ is closed. The concatenation of $p^{l}$ $p^{m}$ and $p^{r}$ may not be simple, so we let $a$ be the first point of $p^{l}$ on $p^{m}$, and $b$ be the last point of $p^{r}$ on $p^{m}$. By the choice of $x'$ and $y'$, $p^{\ell}$ does not intersect $p^{r}$, so we can obtain a simple path by following $p^{\ell}$ from $x$ to $a$, $p^{m}$ from $a$ to $b$, and $p^{r}$ from $b$ to $y$. Call this concatenation $q$.

Let $z$ be a point on $q$ such that $d_{S}^{\omega}(z,Cone^{\omega}_{S}(T)) = d > 0.$ Using lemma 5.8, we can find a path $q'$ such that $q'$ is a $C$ bi-Lipschitz path which is a concatenation of limit geodesics, and the Hausdorff distance between $q$ and $q'$ is less than $\frac{d}{2}$. Thus, there is a point $z'$ on $q'$ such that $d_{S}^{\omega}(z,z') \leq d/2,$ so $d_{S}^{\omega}(z',Cone^{\omega}_{S}(T)) \geq d/2.$ 

Finally we can apply Lemma 5.9 to this new path $q'$ to get that $q' = \lim^{\omega}(q_{n})$ with each $q_{n}$ being a $C'$ bi-Lipschitz path starting and ending in $T$. Thus, as $T$ is Morse, each path is in some fixed neighborhood of $T$. This implies that $q = \lim^{\omega}(q_n)$ is entirely contained in $Cone^{\omega}_{S}(T)$, a contradiction.

Thus, if $T$ is Morse in $S$, then $Cone^{\omega}_{S}(T)$ is strongly convex in $Cone^{\omega}(S)$.
\end{proof}

\begin{defn}
A subgroup $H$ of a group $G$ with finite generating set $X$ is called $\textit{strongly quasi-convex}$ if it is Morse as a subspace of the Cayley graph $G$ with respect to $X$.
\end{defn}

Note that if $H$ is a subgroup of $G$, then for any two points $(h_{i})^{\omega},(k_{i})^{\omega}$ in $Cone^{\omega}_{G}(H)$ there exists an isometry of $Cone^{\omega}(G)$ fixing $Cone^{\omega}_{G}(H)$ which sends $(h_{i})^{\omega}$ to $(k_{i})^{\omega}$. Thus, we can combine the previous two results to give:

\begin{thm}
A subgroup $H$ of a group $G$ is strongly quasi-convex if and only if $Cone^{\omega}_{G}(H)$ is strongly convex in $Cone^{\omega}(G)$ for all ultrafilters $\omega$.
\end{thm}

We conclude by proving a large class of groups cannot contain infinite infinite index strongly quasi convex subgroups.
\begin{thm}
If a path connected metric space $S$ contains a proper closed strongly convex subspace $T$ consisting of more than one point, then $S$ contains a cut point.
\end{thm}

\begin{figure}[h]
\centering
\def\svgwidth{7cm}
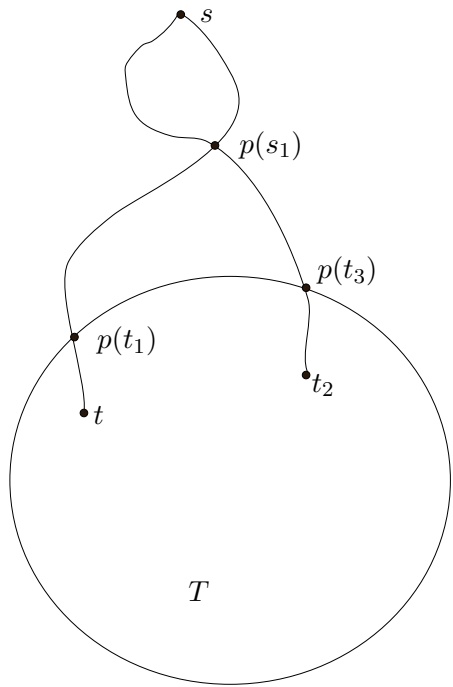
\caption{Theorem 5.13}
\end{figure}

\begin{proof}
Let $s \in S \setminus T$, and let $t \in T$. Let $p \colon [0,\ell] \rightarrow S$ be a simple path connecting $s$ and $t$. Let $t_{1} = \min\{a \in [0,\ell] \mid p(a) \in T\}$. This is well-defined as $T$ is closed. We will show that $p(t_{1})$ is a cut point. Let $t_{2} \neq p(t_{1})$ be a point in $T$. If $p(t_{1})$ is not a cut point, then there exists a path $p' \colon [0,k]$ connecting $s$ and $t_{2}$ such that $p(t_{1})$ is not on $p'$. Let $t_{3} = \min \{a \in [0,k] \mid p'(a) \in T\}$. Let $s_{1} = \max\{a \in [0,t_{1}] \mid p(s_{1}) \in p'\}$ Create a simple path by following $p$ from $t_{1}$ to $s_{1}$ and then following $p'$ from $s_{1}$ to $t_{2}$. This is a simple path connecting two points of $T$ that is not entirely contained in $T$, a contradiction.
\end{proof}

Sapir and Drutu\cite{trees} proved the following theorem.

\begin{thm}
If $G$ is a non-virtually cylic group satisfying a law, then no asymptotic cone of $G$ contains a cut point.
\end{thm}

If $H$ is an infinite, infinite index subgroup of a finitely-generated group $G$, then it is easy to see that $Cone^{\omega}_{G}(H)$ is a proper subspace of $Cone^{\omega}(G)$ that consists of more than one point. Thus, we can combine the previous two results to get the following corollary.

\begin{cor}
If $G$ is a finitely-generated group containing a non-degenerate strongly quasi-convex subgroup $H$, then $G$ does not satisfy a law.
\end{cor}
\bibliographystyle{plain}
  \bibliography{acones}
\end{document}